\documentclass{article}

 \catcode`@=11
\usepackage{amsthm,amssymb,amsfonts,mathrsfs}
\usepackage{amsmath}
\catcode`@=11 \@addtoreset{equation}{section}

\catcode`@=12

\newtheorem{Theorem}{Theorem}[section]
\newtheorem{Proposition}{Proposition}[section]
\newtheorem{Lemma}{Lemma}[section]

\theoremstyle{definition}

\newtheorem{Definition}{Definition}[section]

\newtheorem{Assumptions}{Hypothesis}[section]

\def\R{{\mathbb{R}}}

\def\bs{\boldsymbol}

\def\cA{{\mathcal{A}}}

\def\cU{{\mathcal{U}}}

\def\t\cU{{\widetilde{{\mathcal{U}}}}}

\newcommand\norm[1]{\left\lVert#1\right\rVert}

\def\ds{\displaystyle}

\title {A stability result for a degenerate beam equation}
\author{{{\sc Alessandro Camasta}\thanks{The author is a member of the  {\it Gruppo Nazionale per l'Analisi Ma\-te\-matica, la Probabilit\`a e le loro Applicazioni (GNAMPA)} of the Istituto Nazionale di Alta Matematica (INdAM), a member of {\it UMI ``Modellistica Socio-Epidemiologica (MSE)''} and he is partially supported by PRIN 2017-2019 {\it Qualitative and quantitative aspects of nonlinear PDEs.} He is also supported by the project {\it Mathematical models for interacting dynamics on networks (MAT-DYN-NET)
CA18232  and by the GNAMPA project 2023 {\em Modelli differenziali per l'evoluzione del clima e i suoi impatti}.}}}\\
	Department of Mathematics\\ University of Bari Aldo Moro\\
	Via
	E. Orabona 4\\ 70125 Bari - Italy\\ e-mail: alessandro.camasta@uniba.it\\
	{\sc Genni Fragnelli}\thanks{The author is a member of the  {\it Gruppo Nazionale per l'Analisi Ma\-te\-matica, la Probabilit\`a e le loro Applicazioni (GNAMPA)} of the Istituto Nazionale di Alta Matematica (INdAM), a member of {\it UMI ``Modellistica Socio-Epidemiologica (MSE)''} and is supported by FFABR {\it Fondo per il finanziamento delle attivit\`a base di ricerca} 2017, by  PRIN 2017-2019 {\it Qualitative and quantitative aspects of nonlinear PDEs} and by the
HORIZON$_-$EU$_-$DM737 project 2022 {\it COntrollability of PDEs in the Applied Sciences (COPS)} at Tuscia University. She is also supported by the project {\it Mathematical models for interacting dynamics on networks (MAT-DYN-NET)
CA18232  and by the GNAMPA project 2023 {\em Modelli differenziali per l'evoluzione del clima e i suoi impatti}.}
\\
This paper originated during the {\it ACIPDif 21 School-Workshop on Analysis, Control $\&$ Inverse Problems for Diffusive
Systems with Application to Natural and Social Sciences}, Bari 18-22/07/22, within the Project Horizon Europe Seeds {\it STEPS: STEerability and controllability of PDES in Agricultural and Physical models.
}}\\
	Department of Ecology and Biology\\ Tuscia University\\
	Largo dell'Universit\`a, 01100 Viterbo - Italy\\ e-mail: genni.fragnelli@unitus.it}

\date{}

\begin{document}
	
	\maketitle
	
	\begin{abstract}
We consider a beam equation in presence of a leading degenerate operator which is not in divergence form. We
impose clamped conditions where the degeneracy
occurs and dissipative conditions at the other endpoint. We provide some
conditions for the uniform exponential decay of solutions for the associated
problem.
\end{abstract}

	\noindent 2000AMS Subject Classification: 35L80, 93D23, 93D15

	\section{Introduction}
This paper is devoted to study the stabilization of a boun\-dary degenerate problem with dissipative conditions. In particular, we consider the following problem
	\begin{equation}\label{(P)}
		\begin{cases}
			y_{tt}(t,x)+a(x)y_{xxxx}(t,x)=0, &(t,x)\in Q_T,\\
			y(t,0)=0,\,\,y_x(t,0)=0, &t\in (0,T),\\
			\beta y(t,1)-y_{xxx}(t,1)+y_t(t,1)=0, &t \in (0,T),\\
			\gamma y_x(t,1)+y_{xx}(t,1)+y_{tx}(t,1)=0, &t \in (0,T),\\
			y(0,x)=y_0(x),\,\,y_t(0,x)=y_1(x),&x\in(0,1),
		\end{cases}
	\end{equation}
where $Q_T:=(0,T) \times (0,1)$, $T>0$, $\beta$ and $\gamma$ are non negative constants  and the function $a$ is such that $a(0)=0$ and $a (x) >0$ for all $x \in (0,1]$. 

Problems similar to \eqref{(P)} are considered in several papers (see, for example, \cite{behn}, \cite{biselli}, \cite{chen1}, \cite{chen2}, \cite{chen}, \cite{coleman}, \cite{sandilo}).
In particular, in \cite{chen1} and \cite{chen} the following Euler-Bernoulli beam equation is considered
\begin{equation}\label{EI}
my_{tt}(t,x) + EIy_{xxxx}(t,x)=0, \quad x \in (0,1), \;t>0,
\end{equation}
with clamped conditions at the left end
\begin{equation}\label{co0}
y(t,0)=0, \quad y_x(t,0)=0,
\end{equation}
and with dissipative conditions at the right end
\begin{equation}\label{shear}
\begin{cases}
 -EIy_{xxx}(t,1)+\mu_1y_t(t,1)=0, & \mu_1 \ge0,\\
EIy_{xx}(t,1)+\mu_2y_{tx}(t,1)=0, & \mu_2 \ge0.
\end{cases}
\end{equation}
Here  $m$ is the mass density per unit length and
$EI$ is the flexural rigidity coefficient. Moreover,  the following variables have engineering meanings:
$y$ is the vertical displacement, $y_t$ is the velocity, $y_x$ is the rotation, $y_{tx}$ is the angular velocity, $-EIy_{xx}$ is the bending moment and $-EIy_{xxx}$ is the shear.  In particular, the boundary conditions \eqref{shear}  mean that the shear is proportional to the velocity and the bending moment is negatively proportional to the angular moment. Observe that  if we consider $\beta=\gamma=0$ in \eqref{(P)}, then we have boundary conditions analogous to those in \eqref{shear}. Thus, the dissipative conditions at $1$ are not surprising. We remark that the condition $\beta, \gamma \ge 0$ are necessary to study the well posedness of the problem and to prove the equivalence among all the norms introduced in the paper and that are crucial to obtain the stability result.

The qualitative behaviour of  \eqref{EI}-\eqref{shear} is studied in \cite{chen2}, where it is shown that if $\mu_1^2 >0$ and $\mu_2^2 \ge0$, then the energy of vibration of the beam
\[
E(t)= \frac{1}{2} \int_0^1(m y_t^2(t,x) + EI y_{xx}^2(t,x))dx
\]
decays exponentially in an uniform way:
\begin{equation}\label{decay}
E(t) \le k e^{-\mu t}E(0) 
\end{equation}
for some $k, \mu >0$. Actually, in \cite{chen2} the authors study the stabilization and the control of composite
beams or serially connected beams. Boundary exact controllability on linear beam problems has been also studied in \cite{bugariu}, \cite{krabs}-\cite{rao} and the references therein. In all the previous papers the equation is always non degenerate. However, there are some papers where the equation is degenerate in the sense that a {\it degenerate damping} appears in the equation of \eqref{EI} (see, for example, \cite{cavalcanti}, \cite{Cong}, \cite{narciso}). The first paper where the equation is degenerate in the sense that the fourth order operator degenerates in a point as in \eqref{(P)} is \cite{CF_Beam}, where the boundary controllability is considered.
Hence, this is the first paper where the \textit{stability} for a beam equation governed by a  {\it degenerate} fourth order operator is considered. As for second order degenerate operators (see \cite{alcale}, \cite{BFM wave eq} or \cite{Stability_Genni_Dimitri}), we consider for the function $a$ two cases: the weakly degenerate case and the strongly degenerate one, according to the following two definitions:
\begin{Definition}\label{Def1}
A function $a$ is {\it weakly degenerate at $0$}, $(WD)$ for short, if $a\in\mathcal{C}[0,1]\cap\mathcal{C}^1(0,1]$, $a(0)=0$, $a>0$ on $(0,1]$ and if
	\begin{equation}\label{sup}
		K:=\sup_{x\in (0,1]}\frac{x|a'(x)|}{a(x)},
	\end{equation}
	then $K\in (0,1)$.
\end{Definition}
\begin{Definition}\label{Def2}
A function $a$ is {\it strongly degenerate at $0$}, $(SD)$ for short,  if $a\in\mathcal{C}^1[0,1]$, $a(0)=0$, $a>0$ on $(0,1]$ and in \eqref{sup} we have $K\in [1,2)$.
\end{Definition}

Hence, the main feature in this paper is that $a$ degenerates at $x = 0$ and, as a consequence, classical methods cannot be used directly
to study such a problem and a different approach is needed. In particular, we take inspiration from the technique considered in \cite{alcale} or in \cite{Stability_Genni_Dimitri}, where the stability for a degenerate wave equation in divergence or in non divergence form is studied. Clearly, in this paper the presence of a degenerate fourth order operator brings to more difficulties with respect to the ones for the second order case. These difficulties are related to some new terms that we have to face; for example we have to estimate from above the sum $\int_s^T y^2(t,1)dt+\int_s^Ty_x^2(t,1)dt$ for every $0<s<T$ using the energy associated to the original problem. This is done in Proposition \ref{Prop 3.3} thanks to a suitable fourth order variational problem (introduced in Proposition \ref{prob variazionale}).
Observe that in \cite{BFM wave eq} the authors considered only the controllability for the second order problem in non divergence form, so the technique is completely different form the one used here.

The paper is organized in the following way: in Section \ref{sezione2} we introduce the functional framework crucial to study the well posedness of the problem (see Theorem \ref{Theorem regol}) and we derive some essential results needed for the following. In Section \ref{Section 3} we consider the energy associated to system (\ref{(P)}) and we prove some estimates that allows us to prove the main result of the paper, Theorem \ref{teoremaprincipale}, according to which also when $a$ degenerates in the sense of Definition \ref{Def1} or \ref{Def2}, then the energy satisfies \eqref{decay}, as in  \cite{chen2} for the non degenerate case. We finally underline that with our technique we treat the case  $\beta=0$ and/or $\gamma=0$,  studied in \cite{chen2} for the non degenerate case. The paper concludes with a final section on conclusions and open problems.

\noindent Notations:\\
$'$ denotes the derivative of a function depending on the real space variable $x$; \\
$\dot\ $ denotes the derivative of a function depending on the real time variable $t$;\\
$y_x^2$ or $y_{xx}^2$ means $(y_x)^2$ or $(y_{xx})^2$, respectively.

\section{Preliminary results and well posedness}\label{sezione2}
In this section we introduce the functional spaces needed to study the well posedness of \eqref{(P)}. 
Following \cite{CF}, \cite{CF_Neumann}, \cite{CF_Wentzell} or \cite{CF_Beam} (see also \cite{Libro Genni Dimitri}), let us consider the following weighted Hilbert spaces with the related inner products:
\begin{equation*}
	L^2_{\frac{1}{a}}(0, 1):=\biggl \{u\in L^2(0, 1):\int_{0}^{1}\frac{u^2}{a}\,dx<+\infty \biggr \},\,\,\,\,\,\,\,\,\left\langle u,v\right\rangle_{L^2_{\frac{1}{a}}(0,1)}:=\int_0^1\frac{uv}{a}dx,
\end{equation*}
for every $u,v \in L^2_{\frac{1}{a}}(0,1)$, and
\begin{equation*}
	H^i_{\frac{1}{a}}(0, 1):= L^2_{\frac{1}{a}}(0, 1)\cap H^i(0, 1), \,\,\,\,\,\,\,\,\langle u,v\rangle_{H^i_{\frac{1}{a}}(0,1)}:=\langle u,v\rangle_{L^2_{\frac{1}{a}}(0,1)}+\sum_{k=1}^{i}\langle u^{(k)}, v^{(k)}\rangle_{L^2(0,1)},
\end{equation*}
for every $u,v \in 	H^i_{\frac{1}{a}}(0, 1)$, $i=1,2$. Obviously the previous inner products induce the related respective norms
$
	\ds\norm{u}^2_{L^2_{\frac{1}{a}}(0, 1)}:= \int_{0}^{1}\frac{u^2}{a}dx$, for all $u\in L^2_{\frac{1}{a}}(0, 1)
$
and $ 
	\ds\norm{u}_{H^i_{\frac{1}{a}}(0, 1)}^2:=\norm{u}^2_{L^2_{\frac{1}{a}}(0, 1)} + \sum_{k=1}^{i}\Vert u^{(k)}\Vert^2_{L^2(0, 1)},$ for all $u\in H^i_{\frac{1}{a}}(0, 1),
$
$i=1,2$.  In addition to the previous ones, we introduce the following important Hilbert spaces:
\begin{equation*}
	H^1_{\frac{1}{a},0}(0, 1):= \biggl \{u\in H^1_{\frac{1}{a}}(0, 1):u(0)=0 \biggr \} \quad \text{ and}
\end{equation*}
\begin{equation*}
	H^2_{\frac{1}{a},0}(0, 1):= \biggl \{u\in H^1_{\frac{1}{a},0}(0, 1)\cap H^2(0,1):u'(0)=0 \biggr \},
\end{equation*}
with the norms $\|\cdot\|_{H^i_{\frac{1}{a}}(0, 1)}$, $i=1,2$.
Observe that  the norm $\|\cdot\|_{H^i_{\frac{1}{a}}(0, 1)}$ is equivalent to $\|\cdot\|_i$ in $H^i_{\frac{1}{a}}(0, 1)$, where
$
\| u \|_i ^2:= \|u\|_{L^2_{\frac{1}{a}}(0, 1)}^2+ \|u^{(i)}\|_{L^2(0,1)}^2, \quad i=1,2,
$
for all $u \in H^i_{\frac{1}{a}}(0, 1)$
(see, e.g., \cite{CF_Wentzell}). If $i=1$ the previous assertion is clearly true. 
Moreover, assuming  an additional hypothesis on the function $a$, one can prove that the previous norms are equivalent to the next one
$
	\|u\|_{i, \sim }:= \|u^{(i)}\|_{L^2(0,1)}, $ for all $u\in H^i_{\frac{1}{a}}(0, 1),
	$
	$i=1,2$. In particular, the next proposition holds.
\begin{Assumptions}\label{ipo1}
	The function $a$ is continuous in $[0,1]$, $a(0)=0$, $a>0$ on $(0,1]$ and there exists $K>0$ such that the function
	\begin{equation}\label{crescente}
		x\mapsto\frac{x^K}{a(x)}
	\end{equation}
	is non decreasing in a right neighbourhood of $x=0$.
\end{Assumptions}

Thus, under the previous hypothesis, one can prove the next equivalence.
\begin{Proposition}\label{norms} Assume Hypothesis \ref{ipo1}. Then the norms $\|\cdot\|_{H^i_{\frac{1}{a}}(0, 1)}$, $\| \cdot\|_i$ and
$
	\|\cdot\|_{i, \sim }$
	are equivalent in $H^i_{\frac{1}{a},0}(0, 1)$.  In particular, there exists a constant $C_{HP}>0$ such that
	$\ds
	\|u\|^2_1 \le (C_{HP}+1)\|u\|^2_{1, \sim}
	$
	for all $u \in H^1_{\frac{1}{a},0}(0, 1)$ and
	\begin{equation}\label{x}
	\|u\|^2_2 \le \left(4 C_{HP}+1\right)\|u\|^2_{2, \sim}
	\end{equation}
	for all $u \in H^2_{\frac{1}{a},0}(0, 1)$.
\end{Proposition}
\begin{proof} 
By \cite[Proposition 2.6]{cfr1}, one has that there exists $C>0$ such that
\begin{equation}\label{HP}
\int_0^1\frac{u^2}{a}dx \le C \int_0^1 (u')^2dx,
\end{equation}
for all  $u \in H^1_{\frac{1}{a},0}(0, 1)$. Let
\begin{equation}\label{CHP}
C_{HP} \text{ be the best constant of \eqref{HP}}.
\end{equation}
Thus the thesis follows immediately if $i=1$. Now, assume $i=2$ and fix $u \in H^2_{\frac{1}{a},0}(0, 1)$. Proceeding as for $i=1$ and applying the classical Hardy's inequality to $z:= u'$  (observe that $z \in H^1_{\frac{1}{a},0}(0, 1)$), we have
\[
\begin{aligned}
\int_0^1\frac{u^2}{a}dx &\le C_{HP} \int_0^1 (u')^2dx \le C_{HP}\int_0^1 \frac{z^2}{x^2}dx \\
&\le 4  C_{HP}\int_0^1 (z')^2dx= 4 C_{HP}\int_0^1 (u'')^2dx= 4C_{HP}\|u\|^2_{2, \sim}
\end{aligned}
\]
 and the thesis follows.
\end{proof}
Observe that in Hypothesis \ref{ipo1} we require only continuity on $a$; however, if $a$ is (WD) or (SD), then  the monotonicity property \eqref{crescente} holds globally in $(0,1]$ and
$
	\ds\lim_{x\to 0}\frac{x^\gamma}{a(x)}=0
$
for all $\gamma >K$. 
Thus, assuming in the rest of the paper that $a$ is (WD) or (SD), we can use indifferently $\|\cdot\|_i$ or $\|\cdot\|_{i, \sim }$ in place of $\|\cdot\|_{H^i_{\frac{1}{a}}(0, 1)}$, $i=1,2$.
 Moreover, the following Gauss-Green formulas will be crucial. Setting
\[
\mathcal{W}(0,1):=\Bigl \{u\in H^2_{\frac{1}{a}}(0, 1):au''''\in L^2_{\frac{1}{a}}(0, 1)\Bigr \},\]
we have:

\begin{Lemma}[\cite{CF_Neumann}]\label{Green}For all $(u,v)\in \mathcal{W}(0,1)\times H^2_{\frac{1}{a}}(0, 1)$ one has
	\begin{enumerate}
	\item if $a$ is (WD), then  $\ds
		\int_{0}^{1}u''''v\,dx=[u'''v]^{x=1}_{x=0}-[u''v']^{x=1}_{x=0}+\int_{0}^{1}u''v''dx;$
\item if $a$ is (SD), then
	$\ds
			\int_{0}^{1}u''''v\,dx=u'''(1)v(1)-[u''v']^{x=1}_{x=0}+\int_{0}^{1}u''v''dx.
$
\end{enumerate}
	\end{Lemma}
Finally, to prove the well posedness of \eqref{(P)}, we need to introduce the last Hilbert space $
	\mathcal{H}_0:=H^2_{\frac{1}{a},0}(0,1)\times L^2_{\frac{1}{a}}(0,1),
$
endowed with the inner product
\begin{equation*}
	\langle (u,v),(\tilde{u},\tilde{v})\rangle_{\mathcal{H}_0}:=\int_{0}^{1}u''\tilde{u}''dx+\int_{0}^{1}\frac{v\tilde{v}}{a}dx+\beta u(1)\tilde{u}(1)+\gamma u'(1)\tilde{u}'(1)
\end{equation*}
 and with the norm
\[	\|(u,v)\|^2_{\mathcal{H}_0}:=\int_{0}^{1}(u'')^2dx+\int_{0}^{1}\frac{v^2}{a}dx+\beta u^2(1)+\gamma (u'(1))^2
\]
for every $(u,v), \;(\tilde{u},\tilde{v})\in\mathcal{H}_0$, where $\beta, \gamma \ge 0$. 
Using $	H^2_{\frac{1}{a},0}(0, 1)$, it is possible to define the operator $A: D(A) \subset L^2_{\frac{1}{a}}(0,1)\to L^2_{\frac{1}{a}}(0,1)$ by 
$
	Au:=au'''' $, for all $u \in
	D(A):=\left\{u\in H^2_{\frac{1}{a},0}(0, 1): au''''\in L^2_{\frac{1}{a}}(0, 1) \right\}
$ and
the matrix operator $\mathcal{A}:D(\mathcal{A})\subset\mathcal{H}_0\to \mathcal{H}_0$ given by
\[	\mathcal{A}:=\begin{pmatrix}
	0 & Id \\
	-A & 0
\end{pmatrix}\]
with domain
\begin{equation*}
	\begin{aligned}
		D(\mathcal{A}):=\{(u,v)\in D(A)\times H^2_{\frac{1}{a},0}(0,1): &\beta u(1)-u'''(1)+v(1)=0,\\ &\gamma u'(1)+u''(1)+v'(1)=0 \}.
	\end{aligned}
\end{equation*}
Thanks to $D(A)$ and $H^2_{\frac{1}{a},0}(0, 1)$ one can prove a simpler Gauss-Green formula. In particular, the two formulas in Lemma \ref{Green} become
\begin{equation}\label{GF1}
\int_0^1 u''''v dx= u'''(1)v(1)-u''(1)v'(1) +\int_0^1 u''v''dx
\end{equation}
for all $(u,v) \in D(A) \times H^2_{\frac{1}{a},0}(0, 1)$.
Thanks to \eqref{GF1} one can prove the next theorem that contains the main properties of the operator $(\mathcal{A},D(\mathcal{A}))$.  Since the proof is similar to the one of  \cite{CF_Beam} or \cite{Stability_Genni_Dimitri}, we omit it.

\begin{Theorem}\label{generator}
Assume $a$ (WD) or (SD). Then the operator $(\cA, D(\cA))$ is non positive with dense domain and generates a contraction semigroup  $(S(t))_{t \ge 0}$. 
\end{Theorem}
As in \cite{CF_Beam}, using the operator $(\mathcal A, D(\mathcal A))$,   we can rewrite (\ref{(P)}) as a Cauchy problem. Hence, by \cite[Propositions 3.1 and 3.3]{daprato}, one has the following well posedness result. 
\begin{Theorem}\label{Theorem regol}
	Assume $a$ (WD) or (SD).
	If $(y_0,y_1)\in\mathcal{H}_0$, then there exists a unique mild solution
	$
		y\in \mathcal{C}^1([0,+\infty);L^2_{\frac{1}{a}}(0,1))\cap \mathcal{C}([0,+\infty);H^2_{\frac{1}{a},0}(0,1))
$
of (\ref{(P)}) which depends continuously on the initial data $(y_0,y_1)\in \mathcal{H}_0$. Moreover, if $(y_0,y_1)\in D(\mathcal{A})$, then the solution $y$ is classical, in the sense that
$
	y\in \mathcal{C}^2([0,+\infty);L^2_{\frac{1}{a}}(0,1))\cap \mathcal{C}^1([0,+\infty);H^2_{\frac{1}{a},0}(0,1))\cap \mathcal{C}([0,+\infty);D(A))
$
and the equation of (\ref{(P)}) holds for all $t\ge 0$.
\end{Theorem}

The last result of this section, crucial in the following, is given by the next proposition.
\begin{Proposition}\label{prob variazionale}
	Assume $a$ (WD) or (SD) and consider $\beta,\gamma \ge0$. Define
	\begin{equation*}
		|||z|||^2:=\int_0^1(z'')^2dx+\beta z^2(1)+\gamma (z'(1))^2
	\end{equation*}
for all $z\in H^2_{{\frac{1}{a}},0}(0,1)$. Then the norms $|||\cdot |||$ and $||\cdot ||_{2, \sim}$ are equivalent in $H^2_{{\frac{1}{a}},0}(0,1)$. Moreover, for every $\lambda,\mu\in\mathbb{R}$, the variational problem
\begin{equation*}
	\int_0^1z''\varphi''dx+\beta z(1)\varphi(1)+\gamma z'(1)\varphi'(1)=\lambda\varphi(1)+\mu\varphi'(1), \quad \forall \; \varphi \in H^2_{{\frac{1}{a}},0}(0,1),
\end{equation*}
admits a unique solution $z\in H^2_{{\frac{1}{a}},0}(0,1)$ which satisfies the estimates
\begin{equation}\label{eroeferreo}
	\norm{z}^2_{L^2_{\frac{1}{a}}(0, 1)}\le \left(4 C_{HP}+ 1\right)(|\lambda| +|\mu |)^2\,\,\,\,\,\,\text{ and }\,\,\,\,\,\,	|||z|||^2\le (|\lambda| +|\mu |)^2,
\end{equation}
where $C_{HP}$ is the constant introduced in \eqref{CHP}.
In addition $z\in D(A)$ and solves
\begin{equation}\label{falenaferrea}
		\begin{cases}
		Az=0, \\
		\beta z(1)-z'''(1)=\lambda, \\
		\gamma z'(1)+z''(1)=\mu.
	\end{cases}
\end{equation}
\end{Proposition}
\begin{proof} 
As a first step observe that for all $u \in H^2_{\frac{1}{a},0}(0,1)$ one has
\begin{equation}\label{u(x)}
	|u(x)|=\biggl |\int_0^xu'(t)dt\biggr |=\biggl |\int_0^x\int_0^tu''(\tau)d\tau\, dt\biggr |\le \|u''\|_{L^2(0,1)}=\|u\|_{2, \sim}
\end{equation}
and
\begin{equation}\label{u'(x)}
	|u'(x)|=\biggl |\int_0^xu''(t)dt\biggr |\le \|u''\|_{L^2(0,1)}=\|u\|_{2, \sim}, 
\end{equation}
for all $x \in [0,1]$.
Thus, $|||\cdot|||$ and $\|\cdot\|_{2, \sim}$ are equivalent. Indeed,
for all $u \in H^2_{\frac{1}{a},0}(0,1)$
\begin{equation}\label{04*}
\|u\|_{2, \sim}^2 = \|u''\|_{L^2(0,1)}^2 \le  |||u|||^2.
\end{equation}
Moreover, \eqref{u(x)} and (\ref{u'(x)}) imply  $\beta u^2(1)\le \beta \|u\|_{2, \sim}^2$ and $\gamma (u'(1))^2\le \gamma \|u\|_{2, \sim}^2$, respectively; hence
$
|||u|||^2\le (1 +\beta +\gamma ) \|u\|_{2, \sim}^2
$
and the claim holds.
Now, consider the bilinear and symmetric form $\Lambda: H^2_{\frac{1}{a},0}(0,1) \times H^2_{\frac{1}{a},0}(0,1) \rightarrow \R$ such that
\[
\Lambda (z, \phi) := \int_0^1 z'' \phi'' dx + \beta z(1)\phi(1)+\gamma z'(1)\phi'(1).
\]
As in \cite{Stability_Genni_Dimitri}, by \eqref{04*} and thanks to \eqref{u(x)} and (\ref{u'(x)}), one can easily prove that $\Lambda$ is coercive and continuous. 
Now, consider the linear functional
\[
\mathcal L( \varphi):= \lambda \varphi(1)+\mu \varphi'(1),
\]
with $\varphi \in H^2_{\frac{1}{a},0}(0,1)$ and $\lambda,\mu\in\mathbb{R}$. Clearly, $\mathcal L$ is continuous and linear. Thus, by the Lax-Milgram Theorem, there exists a unique solution $z \in H^2_{\frac{1}{a},0}(0,1)$ of
\begin{equation}\label{05}
\Lambda (z, \varphi)= \mathcal L (\varphi)
\end{equation}
for all $\varphi \in H^2_{\frac{1}{a},0}(0,1)$.
In particular,
\begin{equation}\label{04}
\int_0^1 (z'')^2dx+ \beta z^2(1)+\gamma (z'(1))^2 =  \ \lambda z(1)+\mu z'(1).
\end{equation}
Concerning the other estimates, by \eqref{u(x)}, \eqref{04*} and \eqref{04}, we have 
$
|||z|||^2 = \lambda z(1)+\mu z'(1) \le (|\lambda| +|\mu |)|||z|||;
$
thus
\begin{equation}\label{02}
|||z|||\le |\lambda| +|\mu |.
\end{equation}
Moreover, by Proposition \ref{norms}, we know that in $ H^2_{\frac{1}{a},0}(0,1)$ the two norms 
$
\|\cdot\|_2$ and $\|\cdot\|_{H^2_{\frac{1}{a},0}(0,1)}$
are equivalent. 
Thus, by \eqref{x},
\[
\begin{aligned}
|||z|||^2 &=  \|z\|_{2, \sim}^2 + \beta z^2(1)+\gamma (z'(1))^2 \ge \|z\|_{2, \sim}^2\\ & \ge \frac{1}{4 C_{HP}+1}\|z\|_2^2 \ge \frac{1}{4 C_{HP}+1}\|z\|^2_{L^2_{\frac{1}{a}}(0,1)}.
\end{aligned}
\]
Thus, by \eqref{02},
$
\|z\|^2_{L^2_{\frac{1}{a}}(0,1)} \le \left(4 C_{HP}+1\right)|||z|||^2 \le \left(4 C_{HP}+ 1\right)(|\lambda| +|\mu |)^2.
$

Now, we will prove that $z$ belongs to $D(A)$ and solves \eqref{falenaferrea}. To this aim,  consider again \eqref{05}; clearly, it holds for every $\varphi \in C_c^\infty(0,1)$, so that 
$
\int_0^1  z''\varphi''dx=0 \mbox{ for all }\varphi \in C_c^\infty(0,1).
$
Thus
$z''$ is constant a.e. in $(0,1)$ (see, e.g., \cite[Lemma 1.2.1]{JLJ}) and so 
$
az'''' =0$ a.e.  in  $(0,1),
$
in particular
$Az= az'''' \in L^2_{\frac{1}{a}}(0,1)$; this implies that $z \in D(A)$.

Now, coming back to \eqref{05} and using \eqref{GF1}, we have
\begin{equation*}
	\begin{aligned}
		&\int_0^1 z''\varphi'' dx+ \beta z(1)\varphi(1)+\gamma z'(1)\varphi'(1)=\lambda \varphi(1)+\mu \varphi'(1)\\ &\Longleftrightarrow -z'''(1)\varphi(1)+(z''\varphi')(1)+ \beta z(1)\varphi(1)+\gamma (z'\varphi')(1) = \lambda \varphi(1)+\mu \varphi'(1)
	\end{aligned}
\end{equation*}
for all $\varphi \in H^2_{\frac{1}{a},0}(0,1)$. Thus,
$
-z'''(1) + \beta z(1)=\lambda$ and $\gamma z'(1)+z''(1)=\mu,
$
that is $z$ solves \eqref{falenaferrea}.
\end{proof}
Observe that if $\beta= 0$ and $\gamma =0$, the norms $|||\cdot|||$ and $\|\cdot \|_{2, \sim}$ coincide and the proof of the estimates given in \eqref{eroeferreo} is simpler.

\section{The stability theorem}\label{Section 3}
In this section we prove the main result of the paper; in particular, we will prove that the energy associated to \eqref{(P)} decreases exponentially under suitable assumptions.
As a first step we will give the next definition. 

\begin{Definition}
Let $y$ be a mild solution of (\ref{(P)}) and define its energy  as
\begin{equation*}
	E_y(t):=\frac{1}{2}\int_0^1 \Biggl (\frac{y^2_t(t,x)}{a(x)}+y^2_{xx}(t,x) \Biggr )dx+\frac{\beta}{2}y^2(t,1)+\frac{\gamma}{2}y_x^2(t,1),\quad\,\,\,\,\,\,\forall\;t\ge 0,
\end{equation*}
\end{Definition}
where we recall that $\beta,\gamma \ge 0$. Thus, if $\beta,\gamma \neq 0$ and $y$ is a mild solution, then
\[
y^2(t,1) \le \frac{2}{\beta}E_y(t)
\quad \text{and} \quad y_x^2(t,1) \le \frac{2}{\gamma}E_y(t);
\]
on the other hand, for all $\beta \ge 0$ and $\gamma \ge 0$, one can conclude that
\[
y^2(t,1) \le  2E_y(t) \quad \text{and} \quad y_x^2(t,1) \le  2E_y(t), 
\]
thanks to \eqref{u(x)} and (\ref{u'(x)}), respectively. Thus, we have
\begin{equation}\label{stimapun3}
y^2(t,1) \le C_\beta E_y(t) \quad  \text{and} \quad y_x^2(t,1) \le C_\gamma E_y(t),
\end{equation}
where
$
C_\beta:= \begin{cases}
\ds\min\left \{2,\frac{2}{\beta}\right\}, &\beta \neq 0,\\
2, &\beta=0
\end{cases}\;
\text{and}\;
C_\gamma:= \begin{cases}
\ds\min\left \{2,\frac{2}{\gamma}\right\}, &\gamma \neq 0,\\
2, &\gamma=0
\end{cases}$.

Now we prove that the energy is non increasing.

\begin{Theorem}\label{teorema energia decr}
	 Assume $a$ (WD) or (SD) and let $y$ be a  classical solution of (\ref{(P)}). Then the energy is non increasing. In particular,
	\[
		\frac{dE_y(t)}{dt}=-y^2_t(t,1)-y^2_{tx}(t,1),\quad\,\,\,\,\,\,\,\forall\;t\ge 0.
	\]
\end{Theorem}
\begin{proof}
Multiplying the equation
$
		y_{tt}+Ay=0
$
by $\ds\frac{y_t}{a}$ and integrating over $(0,1)$, we have
$\ds
0=\frac{1}{2}\int_0^1 \Bigl (\frac{y^2_t}{a}\Bigr )_tdx+\int_0^1 y_{xxxx}\,y_tdx.
$
Using the boundary conditions and \eqref{GF1}, one has 
\begin{equation*}
	\begin{aligned}
		0&=\frac{1}{2}\int_0^1 \Bigl (\frac{y^2_t}{a}\Bigr )_tdx+y_{xxx}(t,1)y_t(t,1)-y_{tx}(t,1)y_{xx}(t,1) +\int_0^1y_{xx}y_{txx}dx\\
		&=\frac{1}{2}\frac{d}{dt}\Biggl (\int_0^1\Bigl (\frac{y^2_t}{a}+y^2_{xx}\Bigr )dx+\beta y^2(t,1)+\gamma y_x^2(t,1)\Biggr )+y^2_t(t,1)+y_{tx}^2(t,1)\\ &=\frac{d}{dt}E_y(t)+y^2_t(t,1)+y_{tx}^2(t,1).
	\end{aligned}
\end{equation*}
Hence $\ds
	\frac{d}{dt}E_y(t)=-y^2_t(t,1)-y^2_{tx}(t,1)\le 0$ $\forall\; t\ge 0
$
and, consequently, the energy $E_y$ associated to $y$ is non increasing.
\end{proof}

In the next propositions we establish some equalities important to obtain the stability result. In this sense the following lemma is crucial.
\begin{Lemma}\label{Lemma_bordo}
	Assume Hypothesis \ref{ipo1}. 
	\begin{enumerate}
		\item
		If   $y \in H^1_{\frac{1}{a},0}(0,1)$, then   $\ds
		\lim_{x\rightarrow 0} \frac{x}{a}y^2(x)=0.$
		\item If $a$ is (SD) and $y\in D(A)$, then
		$
		y''\in W^{1,1}(0,1).
		$
	\end{enumerate}
\end{Lemma}

The previous results are proved in \cite[Lemma 3.2.5]{BFM wave eq} and \cite[Proposition 3.2]{CF_Neumann} (see also \cite[Lemma 3.1]{CF_Beam}), respectively.
We underline that the first point of Lemma  \ref{Lemma_bordo} is proved in \cite[Lemma 3.2.5]{BFM wave eq} for all $y \in L^2 _{\frac{1}{a}}(0,1)\cap H^1_0(0,1)$, however the proof holds also if one takes $y \in H^1_{\frac{1}{a},0}(0,1)$ as in this case. Moreover, the second point of the previous lemma is clearly satisfied if $a$ is (WD); indeed, in this case if $y \in D(A)$, then $y \in W^{4,1}(0,1)$ (see \cite{CF_Neumann} for more details).
\begin{Proposition}\label{Prop 3.1}
	Assume  $a$  (WD)  or (SD). If $y$ is a classical solution of (\ref{(P)}), then 
	\begin{equation}\label{prima uguaglianza}
		\begin{aligned}
			0&=2\int_0^1\Bigl [y_t\frac{x}{a}y_x \Bigr ]^{t=T}_{t=s}dx-\frac{1}{a(1)}\int_s^Ty^2_t(t,1)dt+\int_{Q_s}\frac{y^2_t}{a}\Bigl (1-\frac{xa'}{a}\Bigr )dx\,dt\\
			&+3\int_{Q_s} y_{xx}^2dx\,dt+2\beta\int_s^Ty_x(t,1)y(t,1)dt+2\int_s^Ty_x(t,1)y_t(t,1)dt\\
			&+2\gamma\int_s^Ty_x^2(t,1)dt+2\int_s^Ty_x(t,1)y_{tx}(t,1)dt  -  \int_s^T y_{xx}^2(t,1)dt
		\end{aligned}
	\end{equation}
for every $0<s<T$. Here $Q_s:=(s,T)\times (0,1)$.
\end{Proposition}
\begin{proof}
	Fix $s\in (0,T)$. Multiplying the equation in (\ref{(P)}) by $\displaystyle\frac{xy_x}{a}$ and integrating over $Q_s$, we have 
	\begin{equation}\label{int parti}
		0=\int_{Q_s}y_{tt}\frac{xy_x}{a}dx\,dt+\int_{Q_s}ay_{xxxx}\frac{xy_x}{a}dx\,dt.
	\end{equation}
Clearly the previous integrals are well defined; indeed 
$\ds 
\left|y_{tt}\frac{xy_x}{a}\right|=\left|\frac{y_{tt}}{\sqrt{a}}\frac{x}{\sqrt{a}}y_x\right| $ $\ds \le \frac{1}{a(1)}\left|\frac{y_{tt}}{\sqrt{a}}\right||y_x|\in L^1(0,1)$ and 
$
\ds \left|ay_{xxxx}\frac{xy_x}{a}\right|=\left|\frac{Ay}{\sqrt{a}}\frac{x}{\sqrt{a}}y_x\right|$ $\ds \le \frac{1}{a(1)} \left|\frac{Ay}{\sqrt{a}}\right||y_x|\in L^1(0,1).
$
Integrating by parts the first integral in (\ref{int parti}) and observing that 
$
	\lim_{x\to 0}\frac{x}{a(x)}y^2_t(t,x)=0
$
(Lemma \ref{Lemma_bordo}.1),
we obtain
\begin{equation}\label{intparti1}
\begin{aligned}
&\int_{Q_s}y_{tt}\frac{xy_x}{a}dx\,dt
			=\int_0^1\Bigl [y_t\frac{xy_x}{a} \Bigr ]^{t=T}_{t=s}dx-\frac{1}{2}\int_s^T\Bigl [\frac{x}{a}y_t^2\Bigr ]^{x=1}_{x=0}dt+\frac{1}{2}\int_{Q_s}\Bigl (\frac{x}{a}\Bigr )'y^2_tdx\,dt\\
			&=\int_0^1\Bigl [y_t\frac{xy_x}{a} \Bigr ]^{t=T}_{t=s}dx-\frac{1}{2a(1)}\int_s^T y_t^2(t,1)dt+\frac{1}{2}\int_{Q_s}\frac{y^2_t}{a}\Bigl (1-\frac{xa'}{a}\Bigr )dx\,dt.
		\end{aligned}
\end{equation}
Now, consider the second integral in \eqref{int parti}. We will distinguish between the weakly degenerate case and the strongly one.

 \underline{If $a$ is (WD):} in this case $xy_x \in H^2_{\frac{1}{a},0}(0,1)$. Indeed, one can prove that $xy_x \in L^2_{\frac{1}{a}}(0,1)$ and, as in \cite{CF_Neumann}, one can prove that if $y \in D(A)$, then $y \in \mathcal C^3[0,1]$; thus $xy_x \in H^2(0,1)$.  Moreover, $(xy_x)(t, 0)=0=(xy_x)_x(t, 0).$ Thus, we can apply Lemma \ref{Green}, obtaining
\begin{equation}\label{intparti2}
		\begin{aligned}
		&	\int_{Q_s}xy_xy_{xxxx}dx\,dt=\int_s^Ty_x(t,1)y_{xxx}(t,1)dt-\int_s^Ty_{xx}(t,1)(xy_x)_x(t,1)dt\\ &+\int_{Q_s}(xy_x)_{xx}y_{xx}dx\,dt\\
		&=\int_s^Ty_x(t,1)y_{xxx}(t,1)dt-\int_s^Ty_x(t,1)y_{xx}(t,1)dt	-\frac{1}{2}\int_s^Ty_{xx}^2(t,1)dt\\
		&+ \frac{3}{2}\int_{Q_s} y_{xx}^2dx\,dt\\
		&=\int_s^Ty_x(t,1)[\beta y(t,1)+y_t(t,1)]dt-\int_s^Ty_x(t,1)[-\gamma y_x(t,1)-y_{tx}(t,1)]dt\\
		&-\frac{1}{2}\int_s^Ty_{xx}^2(t,1)dt+ \frac{3}{2}\int_{Q_s} y_{xx}^2dx\,dt.
		\end{aligned}
	\end{equation}
\underline{If $a$ is (SD):} fix $\delta \in (0, 1)$; then, using the boundary conditions, one has
\begin{equation}\label{intparti3}
		\begin{aligned}
&\int_s^T\int_\delta^1 xy_xy_{xxxx}dx\,dt= \int_s^T [xy_xy_{xxx}]_{x=\delta}^{x=1}dt - \int_s^T [(xy_x)_xy_{xx}]_{x=\delta}^{x=1}dt \\
&+ \int_s^T\int_\delta^1 (xy_x)_{xx}y_{xx}dxdt
\\&= \int_s^T (xy_xy_{xxx})(t,1)dt -  \int_s^T (xy_xy_{xxx})(t, \delta)dt- \int_s^T ((xy_x)_xy_{xx})(t,1)dt \\
&+ \int_s^T ((xy_x)_xy_{xx})(t, \delta)dt  + \int_s^T\int_\delta^1 (xy_x)_{xx}y_{xx}dxdt\\
&=\int_s^T (y_xy_{xxx})(t,1)dt\! -  \!\int_s^T (xy_xy_{xxx})(t, \delta)dt\!-\!\int_s^T(y_xy_{xx})(t,1)dt\!-\!  \int_s^T y_{xx}^2(t,1)dt \\
&+ \int_s^T (y_xy_{xx})(t, \delta) dt + \int_s^T(xy_{xx}^2)(t, \delta)dt + \int_s^T\int_\delta^1 (xy_x)_{xx}y_{xx}dxdt.
\end{aligned}
	\end{equation}
By Lemma \ref{Lemma_bordo}, we know that $y_{xx} \in W^{1,1}(0,1)$; thus
$
\lim_{\delta \rightarrow 0} (xy_{xx}^2)(t, \delta) =0
$
and
$
\lim_{\delta \rightarrow 0}  (y_xy_{xx})(t, \delta) =0,
$
being $y_x(t,0)=0$.

Now, we will prove that $\lim_{\delta \rightarrow 0}   (xy_xy_{xxx})(t, \delta)=0$. To this aim,
 it is sufficient to prove that $\displaystyle\exists\lim_{\delta\to 0}\delta y_{xxx}(t,\delta)\in\mathbb{R}$. Thus, we rewrite $\delta y_{xxx}(t,\delta)$ as
\begin{equation}\label{Koraidon}
\begin{aligned}
	\delta y_{xxx}(t,\delta)&=y_{xxx}(t,1)-\int_\delta^1(xy_{xxx}(t,x))_xdx\\ &=y_{xxx}(t,1)-\int_\delta^1y_{xxx}(t,x)dx-\int_\delta^1xy_{xxxx}(t,x)dx.
	\end{aligned}
\end{equation}
Note that 
$ |xy_{xxxx}(t, x)|\!=\!\left|\sqrt{a(x)}y_{xxxx}(t, x)\frac{x}{\sqrt{a(x)}}\right|$ $ \le \!$ $\frac{1}{a(1)} \sqrt{a(x)}|y_{xxxx}(t, x)| \!\in \! L^1(0,1).$ Hence, by the absolute continuity of the integral
\begin{equation}\label{3.26_0}
\displaystyle\lim_{\delta\to 0}\int_\delta^1 xy_{xxxx}(t,x)dx = \int_0^1 xy_{xxxx}(t,x)dx.
\end{equation}
On the other hand,
\begin{equation}\label{3.26_1}
\begin{aligned}
	\int_\delta^1y_{xxx}(t,x)dx&=\int_\delta^1\Biggl (y_{xxx}(t,1)-\int_x^1y_{xxxx}(t,s)ds \Biggr )dx\\
	&=(1-\delta)y_{xxx}(t,1)-\int_\delta^1\int_x^1y_{xxxx}(t,s)ds\,dx.
\end{aligned}
\end{equation}
Now, we estimate the last term in the previous equation
\[
\begin{aligned}
	\int_\delta^1\int_x^1y_{xxxx}(t,s)ds\,dx&=\int_\delta ^1 \int_\delta ^s y_{xxxx}(t,s)dxds=\int_\delta^1y_{xxxx}(t,s)(s-\delta)ds\\
	&=\int_\delta^1sy_{xxxx}(t,s)ds-\delta \int_\delta^1y_{xxxx}(t,s)ds.
\end{aligned}
\]
As before,  $\displaystyle\lim_{\delta\to 0}\int_\delta^1 sy_{xxxx}(t,s)ds = \int_0^1 sy_{xxxx}(t,s)ds$.
Moreover, as in \cite[Proof of Theorem 3.2]{CF_Beam}, one has
$
		0<\delta \int_\delta^1|y_{xxxx}(t,s)|ds$ $\le$ $C\delta^{1- \frac{K}{2}}(1-\delta)^{\frac{1}{2}}\Vert\sqrt{a}y_{xxxx}\Vert_{L^2(0,1)},
$ for a positive constant $C$. Thus, since $K<2$, it follows that $\displaystyle\lim_{\delta\to 0}\delta \int_\delta^1y_{xxxx}(t,s)ds= 0$. Consequently, 
\[
		\lim_{\delta\to 0}	\int_\delta^1\int_x^1y_{xxxx}(t,s)ds\,dx=	\lim_{\delta\to 0}	\int_\delta^1y_{xxxx}(t,s)(s-\delta)ds=int_0^1sy_{xxxx}(t,s)ds
\]
and, by \eqref{3.26_1},
\begin{equation}\label{3.26_2}
		\lim_{\delta\to 0}	\int_\delta^1y_{xxx}(t,x)dx = 	y_{xxx}(t,1)- 	\int_0^1sy_{xxxx}(t,s)ds.
\end{equation}
As a consequence, by \eqref{3.26_0} and \eqref{3.26_2}, coming back to (\ref{Koraidon}), one has that there exists $\lim_{\delta\to 0}\delta y_{xxx}(t,\delta) \in \R
$
and, in particular,
$
\lim_{\delta \rightarrow 0} (xy_{xxx})(t, \delta)=\lim_{\delta \rightarrow 0} (xy_xy_{xxx})(t, \delta)$
$=0$.
Now, consider  the term
$\int_\delta^1 (xy_x)_{xx}y_{xx}dx$.
Clearly,
\[
\begin{aligned}
\int_\delta^1 (xy_x)_{xx}y_{xx}dx &= 2\int_\delta^1 y_{xx}^2dx + \frac{1}{2}\int_\delta^1 x(y_{xx}^2)_xdx\\
&=2\int_\delta^1 y_{xx}^2dx  + \frac{1}{2} (y_{xx}^2)(t,1) - \frac{1}{2} (xy_{xx}^2)(t,\delta)  -  \frac{1}{2}\int_\delta^1 y_{xx}^2 dx.
\end{aligned}
\]
By the absolute continuity of the integral and the boundary conditions, one has
$
\lim_{\delta\to 0}\int_\delta^1 (xy_x)_{xx}y_{xx}dx = \frac{3}{2}\int_0^1 y_{xx}^2dx  + \frac{1}{2} y_{xx}^2(t,1).
$
Thus, by \eqref{intparti3}, the previous limits and the boundary conditions, one has the following equality
\begin{equation}\label{intparti4}
\begin{aligned}
&	\int_{Q_s} xy_xy_{xxxx}dx\,dt =\int_s^Ty_x(t,1)[\beta y(t,1)+y_t(t,1)]dt\\
&-\int_s^Ty_x(t,1)[-\gamma y_x(t,1)-y_{tx}(t,1)]dt-\frac{1}{2}\int_s^Ty_{xx}^2(t,1)dt+ \frac{3}{2}\int_{Q_s} y_{xx}^2dx\,dt.
\end{aligned}
\end{equation}
also if $a$ is (SD).
Hence, by \eqref{int parti}, \eqref{intparti1}, \eqref{intparti2} and \eqref{intparti4}, one has
	\[
		\begin{aligned}
			0&=\int_0^1\Bigl [y_t\frac{xy_x}{a} \Bigr ]^{t=T}_{t=s}dx-\frac{1}{2a(1)}\int_s^T y_t^2(t,1)dt+\frac{1}{2}\int_{Q_s}\frac{y^2_t}{a}\Bigl (1-\frac{xa'}{a}\Bigr )dx\,dt\\
			&+\frac{3}{2}\int_{Q_s} y_{xx}^2dxdt+\beta\int_s^Ty_x(t,1)y(t,1)dt+\int_s^Ty_x(t,1)y_t(t,1)dt\\
			&+\gamma\int_s^Ty_x^2(t,1)dt+\int_s^Ty_x(t,1)y_{tx}(t,1)dt  -\frac{1}{2}   \int_s^T y_{xx}^2(t,1)dt.
		\end{aligned}
	\]
Multiplying the previous equality by $2$ we have the thesis.
\end{proof}

As a consequence of the previous equality, we have the next relation. 
\begin{Proposition}
	Assume  $a$  (WD)  or (SD). If $y$ is a classical solution of \eqref{(P)}, then for every $0<s<T$ we have 
\begin{equation}\label{equazione 2}
\int_{Q_s}\frac{y^2_t}{a}\Bigl (\frac{K}{2}+1-\frac{xa'}{a}\Bigr )dx\,dt+\int_{Q_s}y^2_{xx}\Bigl (3-\frac{K}{2}\Bigr )dx\,dt
	=(B.T.),
\end{equation}
where 
\[
\begin{aligned}
(B.T.)&= \frac{K}{2}\int_0^1\Bigl [\frac{yy_t}{a} \Bigr ]^{t=T}_{t=s}dx-2\int_0^1\Bigl [y_t\frac{x}{a}y_x \Bigr ]^{t=T}_{t=s}dx+\frac{K\beta}{2}\int_s^Ty^2(t,1)dt\\
	&+\frac{K}{2}\int_s^Ty(t,1)y_t(t,1)dt+\gamma\Bigl (\frac{K}{2}-2\Bigr )\int_s^Ty_x^2(t,1)dt\\ &+\Bigl (\frac{K}{2}-2\Bigr )\int_s^Ty_x(t,1)y_{tx}(t,1)dt+\int_s^T\frac{y_t^2(t,1)}{a(1)}dt-2\beta\int_s^Ty_x(t,1)y(t,1)dt\\
	&-2\int_s^Ty_x(t,1)y_t(t,1)dt+\int_s^Ty_{xx}^2(t,1)dt.
\end{aligned}
\]
\end{Proposition}
\begin{proof}
Let $y$ be a classical solution of (\ref{(P)}) and fix $s\in (0,T)$. Multiplying the equation in (\ref{(P)}) by $\displaystyle\frac{y}{a}$, integrating over $Q_s$ and using (\ref{GF1}), we obtain
	\begin{equation}\label{eq sommare1}
	\begin{aligned}
		0&=\int_{Q_s}y_{tt}\frac{y}{a}dx\,dt+\int_{Q_s}yy_{xxxx}dx\,dt=\int_0^1\Bigl [y_t\frac{y}{a} \Bigr ]^{t=T}_{t=s}dx\\
		&-\int_{Q_s}\frac{y_t^2}{a}dx\,dt+\int_s^T(yy_{xxx})(t,1)dt-\int_s^T(y_xy_{xx})(t,1)dt+\int_{Q_s}y^2_{xx}dx\,dt.
	\end{aligned}
\end{equation}
Obviously,  all the previous integrals make sense and multiplying (\ref{eq sommare1}) by $\displaystyle\frac{K}{2}$, one has
\begin{equation}\label{eq sommare2}
	\begin{aligned}
		0&=\frac{K}{2}\int_{Q_s}\Bigl (-\frac{y_t^2}{a}+y^2_{xx}\Bigr )dx\,dt+\frac{K}{2}\int_0^1\Bigl [y_t\frac{y}{a} \Bigr ]^{t=T}_{t=s}dx\\
		&+\frac{K}{2}\int_s^Ty(t,1)y_{xxx}(t,1)dt-\frac{K}{2}\int_s^Ty_x(t,1)y_{xx}(t,1)dt.
	\end{aligned}
\end{equation}
By summing (\ref{prima uguaglianza}) and (\ref{eq sommare2}), we get the thesis. 
\end{proof}

By \eqref{equazione 2}, we can get the next estimate.
\begin{Proposition}\label{Prop 3.4}
		Assume $a$ (WD) or (SD) and let $y$ be a classical solution of (\ref{(P)}). Then there exists $\varepsilon_0>0$ such that for every $0<s<T$ 	
		\[
		\begin{aligned}
			&\frac{\varepsilon_0}{2}\int_{Q_s}\Biggl (\frac{y^2_t}{a}+y^2_{xx} \Biggr )dx\,dt\le \Biggl (4\vartheta +\varrho+\frac{C_\gamma}{2}\Biggl (2-\frac{K}{2}\Biggr ) \Biggr )E_y(s)\\
			&+\Biggl (\frac{K\beta}{2}+\frac{K}{4}+\beta\Biggr )\int_s^Ty^2(t,1)dt+(\beta +1+2\gamma^2)\int_s^Ty_x^2(t,1)dt,
		\end{aligned}
	\]
	where $\displaystyle\vartheta :=\max\Biggl\{\frac{4}{a(1)}+KC_{HP},1+\frac{K}{4} \Biggr \}$, $\displaystyle\varrho := \max\Biggl\{2,\frac{K}{4}+1+\frac{1}{a(1)}\Biggr \}$.
\end{Proposition}
\begin{proof} Since by assumption $K<2$, there exists $\varepsilon _0>0$ such that $2-K \ge \varepsilon _0$. Thus,
$
		1-\frac{xa'}{a}+\frac{K}{2}=\frac{(2-K)a+2(Ka-xa')}{2a}\ge \frac{\varepsilon_0}{2}$ and $3-\frac{K}{2}>1-\frac{K}{2}\ge \frac{\varepsilon_0}{2}.$
	Thus, the boundary terms  in (\ref{equazione 2}) can be estimated by below in the following way
\begin{equation}\label{saccoferreo}
	(B.T.)=\int_{Q_s}\frac{y^2_t}{a}\Bigl (1+\frac{K}{2}-\frac{xa'}{a}\Bigr )dx\,dt+\int_{Q_s}y^2_{xx}\Bigl (3-\frac{K}{2}\Bigr )dx\,dt\ge \frac{\varepsilon_0}{2}\int_{Q_s}\Biggl (\frac{y^2_t}{a}+y^2_{xx} \Biggr )dx\,dt.
\end{equation}
Now, we estimate the boundary terms from above. First of all consider the integral
$
\ds	\int_0^1 \Biggl (-2y_t\frac{x}{a}y_x+\frac{K}{2}\frac{yy_t}{a} \Biggr )(\tau,x)dx
$
for all $\tau\in [s,T]$. Using the fact that $\displaystyle \frac{x^2}{a(x)}\le \frac{1}{a(1)}$, together with the classical Hardy's inequality and \eqref{CHP}, one has
\begin{equation*}
	\begin{aligned}
		&\int_0^1 \Biggl (-2y_t\frac{x}{a}y_x+\frac{K}{2}\frac{yy_t}{a} \Biggr )(\tau,x)dx\le\\
		&\le \int_0^1\Biggl (\frac{x^2y_x^2}{a}\Biggr )(\tau,x)dx+\int_0^1\frac{y^2_t}{a}(\tau,x)dx+\frac{K}{4}\int_0^1\frac{y^2_t}{a}(\tau,x)dx+\frac{K}{4}\int_0^1\frac{y^2}{a}(\tau,x)dx\\
		&\le \frac{1}{a(1)}\int_0^1y_x^2(\tau,x)dx+\Biggl (1+\frac{K}{4}\Biggr )\int_0^1\frac{y^2_t}{a}(\tau,x)dx+\frac{KC_{HP}}{4}\int_0^1\frac{x^2y_x^2}{x^2}(\tau,x)dx\\
		&\le \frac{1}{a(1)}\int_0^1\frac{x^2y_x^2}{x^2}(\tau,x)dx+\Biggl (1+\frac{K}{4}\Biggr )\int_0^1\frac{y^2_t}{a}(\tau,x)dx+KC_{HP}\int_0^1y_{xx}^2(\tau,x)dx\\
		&\le \frac{4}{a(1)}\int_0^1y_{xx}^2(\tau,x)dx+\Biggl (1+\frac{K}{4}\Biggr )\int_0^1\frac{y^2_t}{a}(\tau,x)dx+KC_{HP}\int_0^1y_{xx}^2(\tau,x)dx\\
		&\le 2\max\Biggl \{\frac{4}{a(1)}+KC_{HP},1+\frac{K}{4}\Biggr \}E_y(\tau)
	\end{aligned}
\end{equation*}
for all $\tau\in [s,T]$.
Hence
\begin{equation}\label{paldea}
	\begin{aligned}\int_0^1\left[ -2y_t\frac{x}{a}y_x+\frac{K}{2}\frac{yy_t}{a} \right]_{t=s}^{t=T}dx\le 4\max\Biggl \{\frac{4}{a(1)}+KC_{HP},1+\frac{K}{4}\Biggr \}E_y(s).
	\end{aligned}
\end{equation}
Now, by \eqref{stimapun3} and the fact that $K<2$, we have
\begin{equation}\label{primo BT}
\begin{aligned}
	&\gamma\Biggl (\frac{K}{2}-2\Biggr )\int_s^Ty_x^2(t,1)dt+\Biggl (\frac{K}{2}-2\Biggr )\int_s^Ty_x(t,1)y_{tx}(t,1)dt \\ &\le \Biggl (\frac{K}{2}-2\Biggr )\frac{1}{2}(y_x^2(T,1)-y_x^2(s,1))\le \left(2-\frac{K}{2}\right)\frac{1}{2}y_x^2(s,1)\le \Biggl (2-\frac{K}{2}\Biggr )\frac{C_\gamma}{2}E_y(s).
\end{aligned}
\end{equation}
Obviously
\begin{equation}\label{secondo BT}
	\frac{K}{2}\int_s^Ty(t,1)y_t(t,1)dt\le \frac{K}{4}\int_s^Ty^2(t,1)dt+\frac{K}{4}\int_s^Ty_t^2(t,1)dt,
\end{equation}
\begin{equation}\label{terzo BT}
\beta \int_s^T2y_x(t,1)y(t,1)dt\le \beta\int_s^Ty_x^2(t,1)dt+\beta\int_s^Ty^2(t,1)dt
\end{equation}
and
\begin{equation}\label{quarto BT}
 \int_s^T2y_x(t,1)y_t(t,1)dt\le \int_s^Ty_x^2(t,1)dt+\int_s^Ty_t^2(t,1)dt.
\end{equation}
Furthermore, recalling that $\gamma y_x(t,1)+y_{xx}(t,1)+y_{tx}(t,1)=0$,
\begin{equation}\label{quinto BT}
	\int_s^Ty^2_{xx}(t,1)dt\le 2\gamma^2\int_s^Ty_x^2(t,1)dt+2\int_s^Ty_{tx}^2(t,1)dt.
\end{equation}
Hence, by \eqref{saccoferreo}, (\ref{paldea})-(\ref{quinto BT}) and Theorem \ref{teorema energia decr}, we have \begin{equation*}
	\begin{aligned}
	&	\frac{\varepsilon_0}{2}\int_{Q_s}\Biggl (\frac{y^2_t}{a}+y^2_{xx} \Biggr )dx\,dt\le 4\max\Biggl \{\frac{4}{a(1)}+KC_{HP},1+\frac{K}{4}\Biggr \}E_y(s)\\
	&+\Biggl (2-\frac{K}{2}\Biggr )\frac{C_\gamma}{2}E_y(s)+\Biggl (\frac{K}{4}+1+\frac{1}{a(1)}\Biggr )\int_s^Ty_t^2(t,1)dt+2\int_s^Ty_{tx}^2(t,1)dt\\
		&+\Biggl (\frac{K}{4}+\beta +\frac{K\beta}{2}\Biggr )\int_s^Ty^2(t,1)dt+(\beta + 1+2\gamma^2)\int_s^Ty_x^2(t,1)dt
		\end{aligned}
\end{equation*}
\begin{equation*}
	\begin{aligned}
&\le 4\max\Biggl \{\frac{4}{a(1)}+KC_{HP},1+\frac{K}{4}\Biggr \}E_y(s)+\Biggl (2-\frac{K}{2}\Biggr )\frac{C_\gamma}{2}E_y(s)\\
		&+\max\Biggl\{2,\frac{K}{4}+1+\frac{1}{a(1)} \Biggr \}\int_s^T-\frac{d}{dt}E_y(t)dt\\
		&+\Biggl (\frac{K}{4}+\beta +\frac{K\beta}{2}\Biggr )\int_s^Ty^2(t,1)dt+(\beta + 1+2\gamma^2)\int_s^Ty_x^2(t,1)dt\\
		&\le 4\max\Biggl \{\frac{4}{a(1)}+KC_{HP},1+\frac{K}{4}\Biggr \}E_y(s)+\Biggl (2-\frac{K}{2}\Biggr )\frac{C_\gamma}{2}E_y(s)\\
		&+\max\Biggl\{2,\frac{K}{4}+1+\frac{1}{a(1)} \Biggr \}E_y(s)\\
		&+\Biggl (\frac{K}{4}+\beta +\frac{K\beta}{2}\Biggr )\int_s^Ty^2(t,1)dt+(\beta + 1+2\gamma^2)\int_s^Ty_x^2(t,1)dt
	\end{aligned}
\end{equation*}
and the thesis holds.
\end{proof}
In the next proposition, we will find an estimate from above for 
$\ds\int_s^Ty^2(t,1)dt+\int_s^Ty_x^2(t,1)dt.$ To this aim, set 
\begin{equation}\label{nu}
	\nu:=\begin{cases} \ds\frac{\beta \gamma}{2(\beta+\gamma)}, & \text{if } \beta, \gamma >1,\\
	 \ds\frac{ \gamma}{2(1+\gamma)}, & \text{if } \beta\in[ 0, 1], \gamma >1,\\
	\ds\frac{ \beta}{2(\beta+1)}, & \text{if } \beta>1, \gamma \in[ 0, 1],\\
 \ds\frac{1}{4}, & \text{if } \beta,\gamma\in[ 0, 1].\\
	\end{cases}
	\end{equation}
\begin{Proposition}\label{Prop 3.3}
	Assume $a$ (WD)  or (SD). If $y$ is a classical solution of (\ref{(P)}), then for every $0<s<T$ and for every  $\delta \in (0, \nu)$
	 we have
		\[
		\begin{aligned}
			&\int_s^T y^2(t,1)dt+\int_s^Ty_x^2(t,1)dt\le  \frac{2\delta}{C_\delta}\int_s^TE_y(t)dt\\
			&+\frac{1}{C_\delta}\left(2 (1+  (4C_{HP}+1)\left(C_\beta+ C_\gamma \right))+\frac{1}{\delta}\left(8 C_{HP}+ 3\right)\right)E_y(s),
		\end{aligned}
	\]
	where
\[
C_\delta:= \begin{cases} \ds 1- 2\delta\Biggl (\frac{1}{\beta}+\frac{1}{\gamma}\Biggr ), & \text{if } \beta, \gamma >1,\\
\ds 1- 2\delta\Biggl (1+\frac{1}{\gamma}\Biggr ), & \text{if } \beta \in[ 0, 1], \gamma >1,\\
\ds 1- 2\delta\Biggl (\frac{1}{\beta}+1\Biggr ), & \text{if } \beta>1, \gamma \in[ 0, 1],\\
\ds 1- 4\delta, & \text{if } \beta, \gamma \in[ 0, 1].
\end{cases}
\]
\end{Proposition}
\begin{proof}
Set $\lambda =y(t,1)$, $\mu =y_x(t,1)$, where $t \in [s,T]$, and let $z=z(t,\cdot)\in H^2_{{\frac{1}{a}},0}(0,1)$ be the unique solution of
\begin{equation*}
	\int_0^1z_{xx}\varphi''dx+\beta z(t,1)\varphi(1)+\gamma z'(t,1)\varphi'(1)=\lambda\varphi(1)+\mu\varphi'(1),\,\,\,\,\,\,\,\quad\forall\;\varphi\in H^2_{{\frac{1}{a}},0}(0,1).
\end{equation*}
By Proposition \ref{prob variazionale}, $z(t,\cdot)\in D(A)$ for all $t$ and solves
 \begin{equation}\label{fogliaferrea}
 	\begin{cases}
 		Az=0, \\
 		\beta z(t,1)-z_{xxx}(t,1)=\lambda, \\
 		\gamma z_x(t,1)+z_{xx}(t,1)=\mu.
 	\end{cases}
 \end{equation}
Now, multiplying the equation in (\ref{(P)}) by $\displaystyle\frac{z}{a}$ and integrating over $Q_s$, we have
\begin{equation}\label{acquecrespe0}
	\begin{aligned}
		0&=\int_{Q_s}y_{tt}\frac{z}{a}dx\,dt+\int_{Q_s}zy_{xxxx}dx\,dt\\
		&=\int_0^1\Bigl [y_t\frac{z}{a} \Bigr ]^{t=T}_{t=s}dx-\int_{Q_s}\frac{y_tz_t}{a}dx\,dt+\int_s^Tz(t,1)y_{xxx}(t,1)dt-\int_s^Tz_x(t,1)y_{xx}(t,1)dt\\
		&+\int_{Q_s}z_{xx}y_{xx}dx\,dt.
	\end{aligned}
\end{equation}
Hence, (\ref{acquecrespe0}) reads
\begin{equation}\label{acquecrespe}
\begin{aligned}
	&\int_0^1\Bigl [y_t\frac{z}{a} \Bigr ]^{t=T}_{t=s}dx-\int_{Q_s}\frac{y_tz_t}{a}dx\,dt\\
	&=-\int_s^Tz(t,1)y_{xxx}(t,1)dt+\int_s^Tz_x(t,1)y_{xx}(t,1)dt-\int_{Q_s}z_{xx}y_{xx}dx\,dt.
\end{aligned}
\end{equation}
On the other hand, multiplying the equation in (\ref{fogliaferrea}) by $\displaystyle\frac{y}{a}$ and integrating over $Q_s$, we have
$
	\int_{Q_s}z_{xxxx}y\,dx\,dt=0.
$
By (\ref{GF1}), we get
$
\int_{Q_s}z_{xx}y_{xx}dx\,dt=-\int_s^Tz_{xxx}(t,1)y(t,1)dt+\int_s^Ty_x(t,1)z_{xx}(t,1)dt.
$
Substituting in (\ref{acquecrespe}), using the fact that $z_{xxx}(t,1)=\beta z(t,1)-\lambda$, $z_{xx}(t,1)=-\gamma z_x(t,1)+\mu$, $\lambda =y(t,1)$ and $\mu=y_x(t,1)$, we have
\begin{equation*}
\begin{aligned}
&		\int_0^1\Bigl [y_t\frac{z}{a} \Bigr ]^{t=T}_{t=s}dx-\int_{Q_s}\frac{y_tz_t}{a}dx\,dt=-\int_s^Tz(t,1)y_{xxx}(t,1)dt+\int_s^Tz_x(t,1)y_{xx}(t,1)dt\\
		&+\int_s^Tz_{xxx}(t,1)y(t,1)dt-\int_s^Ty_x(t,1)z_{xx}(t,1)dt\\
		&=-\int_s^Tz(t,1)y_{xxx}(t,1)dt+\int_s^Tz_x(t,1)y_{xx}(t,1)dt+\int_s^Ty(t,1)[\beta z(t,1)-\lambda ]dt\\
		&-\int_s^Ty_x(t,1)[-\gamma z_x(t,1)+\mu ]dt\\
		&=\int_s^Tz(t,1)[\beta y(t,1)-y_{xxx}(t,1)]dt\\ &+\int_s^Tz_x(t,1)[y_{xx}(t,1)+\gamma y_x(t,1)]dt-\int_s^Ty^2(t,1)dt-\int_s^Ty_x^2(t,1)dt.
\end{aligned}
\end{equation*}
Then 
\begin{equation}\label{solcoferreo}
\begin{aligned}
	\int_s^Ty^2(t,1)dt+\int_s^Ty_x^2(t,1)dt&=-\int_s^T(y_tz)(t,1)dt-\int_s^T(z_xy_{tx})(t,1)dt\\ &-\int_0^1\Bigl [y_t\frac{z}{a} \Bigr ]^{t=T}_{t=s}dx+\int_{Q_s}\frac{y_tz_t}{a}dx\,dt.
\end{aligned}
\end{equation}
Thus, in order to estimate $\int_s^Ty^2(t,1)dt+\int_s^Ty_x^2(t,1)dt$, we have to consider the four terms in the previous equality. So, by (\ref{eroeferreo}), \eqref{stimapun3} and Theorem \ref{teorema energia decr} we have, for all $\tau\in [s,T]$,
\begin{equation*}
	\begin{aligned}
		\int_0^1\Bigl |\frac{y_tz}{a}(\tau,x) \Bigr |dx&\le  \frac{1}{2}\int_0^1\frac{y^2_t(\tau,x) }{a(x)}dx+\frac{1}{2}\int_0^1\frac{z^2(\tau,x) }{a(x)}dx\\
		&\le \frac{1}{2}\int_0^1\frac{y^2_t(\tau,x) }{a(x)}dx+(4C_{HP}+1)\left(y^2(\tau,1)+y_x^2(\tau,1)\right)\\
		& \le E_y(\tau) +(4C_{HP}+1)\left(C_\beta+ C_\gamma \right)E_y(\tau)\\
& \le (1+  (4C_{HP}+1)\left(C_\beta+ C_\gamma \right))E_y(s).
	\end{aligned}
\end{equation*}
By Theorem \ref{teorema energia decr}, 
\begin{equation}\label{latios}
\Bigl |	\int_0^1\Bigl [\frac{y_tz}{a}\Bigr ]^{t=T}_{t=s} dx\Bigr |\le 2 (1+  (4C_{HP}+1)\left(C_\beta+ C_\gamma \right))E_y(s).
\end{equation}
Moreover, for any $\delta >0$ we have
\begin{equation}\label{colloferreo}
	\int_s^T|(y_tz)(t,1)|dt\le \frac{1}{\delta}\int_s^Ty^2_t(t,1)dt+\delta\int_s^Tz^2(t,1)dt.
\end{equation}
By definition of $|||\cdot|||$,  if $\beta >1$, one has
$
	z^2(t,1)\le \frac{1}{\beta}|||z|||^2\le \frac{1}{\beta}(|\lambda| +|\mu|)^2 \le  \frac{2}{\beta}(y^2(t,1)+y_x^2(t,1));
$
on the other hand, for all $\beta \in [0,1]$, by \eqref{u(x)}, it results
$
z^2(t,1)\le |||z|||^2\le 2(y^2(t,1)+y_x^2(t,1)).
$
Thus, by (\ref{colloferreo}), we have
\begin{equation}\label{latias}
	\begin{aligned}
		\int_s^T|(y_tz)(t,1)|dt&\le  \begin{cases}\ds \frac{1}{\delta}\int_s^Ty^2_t(t,1)dt+ \frac{2\delta}{\beta}\int_s^T(y^2+y_x^2)(t,1)dt,& \text{if } \beta >1,\\
	\ds	\frac{1}{\delta}\int_s^Ty^2_t(t,1)dt+ 2\delta\int_s^T(y^2+y_x^2)(t,1)dt,& \text{if } \beta \in [0,1].
\end{cases}
	\end{aligned}
\end{equation}
In a similar way,  it is possible to find the next estimate
\begin{equation}\label{LATIOS}
	\int_s^T|(z_xy_{tx})(t,1)|dt\le \begin{cases} \ds \frac{1}{\delta}\int_s^Ty^2_{tx}(t,1)dt+ \frac{2\delta}{\gamma}\int_s^T(y^2+y_x^2)(t,1)dt, & \text{if } \gamma >1,\\
	\ds \frac{1}{\delta}\int_s^Ty^2_{tx}(t,1)dt+ 2\delta\int_s^T(y^2+y_x^2)(t,1)dt, & \text{if } \gamma \in[0,1],
	\end{cases}
\end{equation}
being
$
	\bs z_x^2(t,1)\le \frac{1}{\gamma}|||z|||^2\le  \frac{2}{\gamma}(y^2(t,1)+y_x^2(t,1)),
$
if $\gamma >1$ (by definition of $|||\cdot|||$), and
 $
\bs z_x^2(t,1)\le |||z|||^2\le 2(y^2(t,1)+y_x^2(t,1))
$
for all $\gamma \in [0,1]$ (by \eqref{u'(x)} and by the definition of $|||\cdot|||$).
Therefore, summing \eqref{latias} and \eqref{LATIOS} and applying Theorem \ref{teorema energia decr} we obtain
\begin{equation}\label{somma}
\begin{aligned}
&\int_s^T|(y_tz)(t,1)|dt+	\int_s^T|(z_xy_{tx})(t,1)|dt \\
&\le \frac{1}{\delta}\int_s^T -\frac{d}{dt}E_y(t)dt+2\delta\Biggl (\frac{1}{\beta}+\frac{1}{\gamma}\Biggr )\int_s^T(y^2+y_x^2)(t,1)dt\\
&\le  \frac{E_y(s)}{\delta}+2\delta\Biggl (\frac{1}{\beta}+\frac{1}{\gamma}\Biggr )\int_s^T(y^2+y_x^2)(t,1)dt,
\end{aligned}
\end{equation}
if $\beta, \gamma >1$. On the other hand,
\begin{equation}\label{somma1}
\int_s^T|(y_tz)(t,1)|dt+	\int_s^T|(z_xy_{tx})(t,1)|dt\le  \frac{E_y(s)}{\delta}+2\delta\Biggl (1+\frac{1}{\gamma}\Biggr )\int_s^T(y^2+y_x^2)(t,1)dt,
\end{equation}
if $\beta \in[ 0, 1]$ and $\gamma >1$,
\begin{equation}\label{somma2}
\int_s^T|(y_tz)(t,1)|dt+	\int_s^T|(z_xy_{tx})(t,1)|dt\le  \frac{E_y(s)}{\delta}+2\delta\Biggl (\frac{1}{\beta}+1\Biggr )\int_s^T(y^2+y_x^2)(t,1)dt,
\end{equation}
if $\beta >1$ and $\gamma\in[ 0, 1]$,
\begin{equation}\label{somma3}
\int_s^T|(y_tz)(t,1)|dt+	\int_s^T|(z_xy_{tx})(t,1)|dt\le  \frac{E_y(s)}{\delta}+4\delta\int_s^T(y^2+y_x^2)(t,1)dt,
\end{equation}
if $\beta, \gamma\in[ 0, 1]$.
Finally, we estimate the last integral in (\ref{solcoferreo}), i.e. $\ds\int_{Q_s}\Bigl |\frac{y_tz_t}{a}\Bigr |dx\,dt$. To this aim, consider again problem (\ref{falenaferrea}) and differentiate with respect to $t$. Thus
\begin{equation*}
	\begin{cases}
		a(z_t)_{xxxx}=0, \\
		\beta z_t(t,1)-(z_t)_{xxx}(t,1)=y_t(t,1), \\
		\gamma (z_t)_x(t,1)+(z_t)_{xx}(t,1)=(y_x)_t(t,1).
	\end{cases}
\end{equation*}
Clearly, $z_t$ satisfies \eqref{eroeferreo}, in particular
\[\ds 
	\norm{z_t (t)}^2_{L^2_{\frac{1}{a}}\!\!(0, 1)}\!\!\le\left(4 C_{HP}\!+ \!1\right)\!(|y_t(t,1)|+|y_{tx}(t,1)|)^2
	\]
	and \[\ds |||z_t(t)|||^2\le  (|y_t(t,1)|+|y_{tx}(t,1)|)^2.
\]
Thus, for $\delta >0$ we find
\begin{equation}\label{keldeo}
	\begin{aligned}
		\int_{Q_s}\Bigl |\frac{y_tz_t}{a}\Bigr |dx\,dt&\le \delta\int_{Q_s}\frac{y^2_t}{a}dx\,dt+\frac{1}{\delta}\int_{Q_s}\frac{z^2_t}{a}dx\,dt \\
		&\le 2\delta\int_s^TE_y(t)dt+\frac{2}{\delta}\left(4 C_{HP}+ 1\right)\int_s^T(y^2_t(t,1)+ y^2_{tx}(t,1))dt\\
		&=2\delta\int_s^TE_y(t)dt+\frac{2}{\delta}\left(4 C_{HP}+ 1\right)\int_s^T-\frac{d}{dt}E_y(t)dt\\
		&\le 2\delta\int_s^TE_y(t)dt+\frac{2}{\delta}\left(4 C_{HP}+ 1\right)E_y(s).
	\end{aligned}
\end{equation}
Coming back to (\ref{solcoferreo}) and using (\ref{latios}), (\ref{somma}) - \eqref{somma3}, (\ref{keldeo}), we obtain for every  $\delta \in (0, \nu)$
\begin{equation*}
	\begin{aligned}
		C_\delta\int_s^T(y^2(t,1)+y_x^2(t,1))dt&\le  2 (1+  (4C_{HP}+1)\left(C_\beta+ C_\gamma \right))E_y(s)\\
			&+2\delta\int_s^TE_y(t)dt+\frac{1}{\delta}\left(8 C_{HP}+ 3\right)E_y(s),
	\end{aligned}
\end{equation*}
and the thesis follows.
\end{proof}

As a consequence of Propositions \ref{Prop 3.4} and \ref{Prop 3.3}, we can formulate the main result of the paper, whose proof is based on  \cite[Theorem 8.1]{Ko}.
\begin{Theorem}\label{teoremaprincipale}
	Assume $a$ (WD) or (SD) and let $y$ be a mild solution of (\ref{(P)}). Then for all $t>0$ and for all $
	\ds	\delta \in \left(0,\min \left\{\nu, \frac{\varepsilon_0}{C_1}\right\}\right),
$
	\begin{equation}\label{Stabilità}
		E_y(t)\le E_y(0)e^{1-\frac{t}{M}},
	\end{equation}
where
$
M:=\frac{C_2}{\varepsilon_0 - \delta C_1},
$ 
\[C_1:= \frac{2}{C_\delta}\left(\frac{K\beta}{2}+\frac{K}{4}+\beta+\varepsilon_0\frac{\beta}{2}+\beta +1+2\gamma^2+\varepsilon_0\frac{\gamma}{2}\right),\]
\[
	C_2:=4\vartheta +\varrho+\frac{C_\gamma}{2}\Biggl (2-\frac{K}{2}\Biggr )+ C_3
	\]
	and
	\[
	\begin{aligned} C_3:=\ds&\frac{1}{C_\delta}\left( 2\!+\!2  (4C_{HP}+1)\left(C_\beta\!+\! C_\gamma \right)+\frac{1}{\delta}\left(8 C_{HP}+ 3\right)\right)\\&
	\cdot\left(\frac{K\beta}{2}+\frac{K}{4}+\beta+\varepsilon_0\frac{\beta}{2}+\beta +1+2\gamma^2+\varepsilon_0\frac{\gamma}{2}\right).
\end{aligned}
\]
\end{Theorem}

	\begin{proof}
		As a first step, consider $y$ a classical solution of \eqref{(P)} and take $\delta>0$ such that
$
	\ds	\delta < \min \left\{\nu, \frac{\varepsilon_0}{C_1}\right\},
$
	where $\nu$ is the constant defined in \eqref{nu}.
	By definition of $E_y$ and Propositions \ref{Prop 3.4}, \ref{Prop 3.3}, we have
		\[
		\begin{aligned}
&		\varepsilon_0  \int_s^TE_y(t)dt=\frac{\varepsilon_0}{2}\int_{Q_s} \Biggl (\frac{y^2_t(t,x)}{a(x)}+y^2_{xx}(t,x) \Biggr )dxdt\\
		&+\varepsilon_0\frac{\beta}{2}\int_s^Ty^2(t,1)dt+\varepsilon_0\frac{\gamma}{2}\int_s^Ty_x^2(t,1)dt\\
		&\le  \Biggl (4\vartheta +\varrho+\frac{C_\gamma}{2}\Biggl (2-\frac{K}{2}\Biggr ) \Biggr )E_y(s)+\Biggl (\frac{K\beta}{2}+\frac{K}{4}+\beta+\varepsilon_0\frac{\beta}{2}\Biggr )\int_s^Ty^2(t,1)dt\\
		&+\left(\beta +1+2\gamma^2+\varepsilon_0\frac{\gamma}{2}\right)\int_s^Ty_x^2(t,1)dt\\
	&	\le  \Biggl (4\vartheta +\varrho+\frac{C_\gamma}{2}\Biggl (2-\frac{K}{2}\Biggr )+ C_3 \Biggr )E_y(s)+\delta C_1\int_s^TE_y(t)dt.
		\end{aligned}
	\]
This implies $
			\Biggl [\varepsilon_0-\delta C_1 \Biggr ]\int_s^TE_y(t)dt\le C_2E_y(s).
		$
Hence, we can apply \cite[Theorem 8.1]{Ko} with $M:= \frac{C_2}{\varepsilon_0-\delta C_1}$ and \eqref{Stabilità} holds.
	If $y$ is the mild solution of the problem, we can proceed as in \cite{Stability_Genni_Dimitri}, obtaining the thesis.
	\end{proof}

\section{Conclusions and open problems}\label{conclusion}
In this paper we have considered a beam equation  governed by a degenerate operator in non divergence form under clamped conditions at the degeneracy point and dissipative conditions at the other endpoint. In Theorem \ref{teoremaprincipale} we provide some
conditions for the uniform exponential decay of solutions for the associated
Cauchy problem. The same equation under a controllability point of view is considered in \cite{CF_Beam}. On the other hand 
the stability and the controllability for the problem in divergence form is studied recently in \cite{CF_div}. Thus, this paper fits the current lines of research in degenerate problems.

\bibliographystyle{siamplain}

\bibliography{references}

	\end{document}